\documentclass[12pt]{amsart}

\usepackage{amssymb,amsmath,amsthm, url}
\usepackage{comment}
\usepackage{enumitem}
\usepackage[all]{xy}

\allowdisplaybreaks[1]

\newtheorem*{namedtheorem}{\theoremname}
\newcommand{\theoremname}{testing}

\newtheorem{theorem}{Theorem}% [section]

\newtheorem{proposition-definition}[theorem]
{Proposition-Definition}

\newtheorem{lemma}[theorem]{Lemma}

\theoremstyle{definition}

\newtheorem{example}[theorem]{Example}

\newtheorem{remark}[theorem]{Remark}

\theoremstyle{remark}

\renewcommand{\mathcal}{\mathscr}

\newcommand\CC{\mathbb{C}} 
 
\newcommand\GG{\mathbb{G}}

 \newcommand\PP{\mathbb{P}}
\newcommand\QQ{\mathbb{Q}}

 \newcommand\ZZ{\mathbb{Z}}

\newcommand\arr{\ifinner\to\else\longrightarrow\fi}

\newcommand\arrto{\ifinner\mapsto\else\longmapsto\fi}

\newcommand\Aut{\operatorname{Aut}}

\newcommand{\Ind}{\operatorname{Ind}}

\newcommand\Span{\operatorname{Span}}

\def\displaytimes_#1{\mathrel{\mathop{\times}\limits_{#1}}}

\def\displayotimes_#1{\mathrel{\mathop{\bigotimes}\limits_{#1}}}

\newcommand\Bir{\operatorname{Bir}}

\newcommand\rdim{\operatorname{rdim}}

\newcommand\Sym{\operatorname{S}}

\newcommand\Cr{\operatorname{Cr}}

\newcommand{\GL}{\mathrm{GL}}

\newcommand{\ed}{\operatorname{ed}}

\begin{document}

\title[Jordan property and essential dimension]{The Jordan property of Cremona groups and essential dimension}

 \author{Zinovy Reichstein}
\address{Department of Mathematics\\
% 1984 Mathematics Road\\
 University of British Columbia\\
 Vancouver, BC V6T 1Z2\\Canada}
 \email{reichst@math.ubc.ca}
\thanks{Partially supported by
 National Sciences and Engineering Research Council of
 Canada Discovery grant 253424-2017.}

% \date{February 20, 2018}

% 13Axx		General commutative ring theory
% 13A18  	Valuations and their generalizations [See also 12J20]
% 13A50  	Actions of groups on commutative rings; invariant theory [See also 14L24]
% 14E07  View Publications (1980-now) Birational automorphisms, Cremona group and generalizations
% 14L30  View Publications (1980-now) Group actions on varieties or schemes (quotients)
% 20C05  (1973-now) Group rings of finite groups and their modules [See also 16S34]

\begin{abstract} We use a recent advance in birational geometry to prove
new lower bounds on the essential dimension of some finite groups.
\end{abstract}

\maketitle
%\tableofcontents

\section{Introduction}

A abstract group $\Gamma$ is called {\em Jordan} if there exists an integer $j$ such that every finite subgroup $G \subset \Gamma$ has a 
normal abelian subgroup $A$ of index $[G: A] \leqslant j$. This definition, due to V. L. Popov~\cite{popov1, popov2}, 
was motivated by the classical theorem of Camille Jordan~\cite{jordan} which asserts that $\GL_n(k)$ is Jordan, 
and by a theorem of J.-P.~Serre~\cite{serre} which asserts that the Cremona group $\Cr_2(k)$ is also Jordan. Here and throughout this note
$k$ denotes a base field of characteristic $0$. The Cremona group $\Cr_2 = \Bir(\PP^2)$ is the group of birational 
automorphisms of the projective plane. Serre asked whether the higher Cremona groups $\Cr_n = \Bir(\PP^n)$ are Jordan as well. 
Our starting point is the following remarkable theorem of Y.~Prokhorov, C.~Shramov and C.~Birkar, which asserts that
groups of birational isomorphisms of rationally connected varieties of fixed dimension are ``uniformly Jordan". 

\begin{theorem} \label{thm.jordan} {\rm (\cite[Corollary 1.3]{birkar})}
For every positive integer $n \geqslant 1$ there exists a positive integer $j(n)$ with the following property.
Let $G$ be a finite subgroup of the group $\Bir(X)$ of birational automorphisms of an $n$-dimensional rationally connected variety $X$.
Then $G$ has a normal abelian subgroup $A$ such that $[G: A] \leqslant j(n)$.
\end{theorem}

Prokhorov and Shramov~\cite{ps16} proved this theorem assuming the Borisov-Alexeev-Borisov (BAB) conjecture. 
The BAB conjecture was subsequently proved by Birkar~\cite{birkar}. In this note we will deduce some consequences 
of Theorem~\ref{thm.jordan} concerning essential dimension of finite groups. 

Let $G$ be a finite group. Recall that the {\em representation dimension} $\rdim_k(G)$ is the minimal dimension of a faithful 
representation of $G$ defined over $k$, i.e., the smallest positive integer $r$ such that $G$ is isomorphic to 
a subgroup of $\GL_r(k)$. The {\em essential dimension} $\ed_k(G)$ is the minimal dimension of a faithful linearizable $G$-variety
defined over $k$. Here by a faithful $G$-variety we mean an algebraic variety $X$ with a faithful action of $G$.
We say that $X$ is linearizable if there exists a $G$-equivariant dominant rational map $V \dasharrow X$, 
where $V$ is a vector space with a linear action for $G$. 

It is clear from these definitions that
\begin{equation} \label{e.ed-rd} \ed_k(G) \leqslant \rdim_k(G).\end{equation}
We will write $\ed(G)$ and $\rdim(G)$ in place of $\ed_k(G)$ and $\rdim_k(G)$, respectively, when
the reference of $k$ is clear from the context. Equality in~\eqref{e.ed-rd} holds in two interesting cases: 

\begin{itemize}
\item if $G$ is abelian and $k$ contains 
a primitive $e$th root of unity, where $e$ is the exponent of $G$ (see~\cite[Theorem 6.1]{bur}), or 

\item
if $G$ is a $p$-group and $k$ contains
a primitive $p$th root of unity (see~\cite{km08}). 
\end{itemize}

For other finite groups, $\ed(G)$ and $\rdim(G)$ can diverge. 
Our first main result shows that they do not diverge too far, assuming $k$ contains suitable roots of unity.

\begin{theorem} \label{thm1} Let $r(n) = nj(n)$, where $j(n)$ is the Jordan constant from Theorem~\ref{thm.jordan}. 
Suppose $G$ is a finite group of exponent $e$ and the base field $k$ contains a primitive $e$th root of unity. 

(a) If $\ed_k(G) \leqslant n$, then $\rdim_k(G) \leqslant r(n)$. 

(b) Moreover, if $\ed_k(G) \leqslant n$, then $G$ is isomorphic to a finite subgroup 
of $\GG_m^{r(n)} \rtimes \Sym_{r(n)}$, where the symmetric group $\Sym_{r(n)}$ acts on $\GG_m^{r(n)}$ by permuting the factors.
\end{theorem}

To place Theorem~\ref{thm1}(a) into the context of what is currently known about essential dimension of finite groups, 
let us assume for simplicity that $k$ is algebraically closed. Let $G$ be a finite group, 
$p$ be a prime, and $G[p]$ be a Sylow $p$-subgroup of $G$. As we 
mentioned above, $\ed(G[p]) = \rdim(G[p])$ by the Karpenko-Merkurjev theorem~\cite{km08}, and $\rdim(G[p])$ 
can be computed, at least in principle, by the methods of representation theory of finite groups.
This way we obtain a lower bound 
\begin{equation} \label{e.sylow}
\ed(G) \geqslant \max_p \, \rdim(G[p]).
\end{equation}
One can then try to prove a matching upper bound by constructing an explicit $d$-dimensional
faithful linearizable $G$-variety of dimension $\max_p \, \rdim(G[p])$.
In most of the cases where the exact value of $\ed(G)$ is known, it was established using
this strategy. 

There are, however, finite groups $G$ for which the inequality~\eqref{e.sylow} is strict.
All known proofs of stronger lower bounds of the form $\ed(G) > d$
appeal to the classification of finite subgroups of $\Bir(X)$, 
where $X$ ranges over the $d$-dimensional unirational (or rationally connected)
varieties.
Such classifications is available only for $d = 1$ (see~\cite[Chapter 1]{klein}) and $d= 2$, and 
the latter is rather complicated; see~\cite{di}.  
For $d = 3$ there is only a partial classification (see \cite{prokhorov}),
and for $d \geqslant 4$ even a partial classification is currently out of reach.
Lower bounds of the form $\ed(G) > d$ proved by this method (for suitable finite groups $G$), 
can be found 

\begin{itemize}

\smallskip
\item
in~\cite[Theorem 6.2]{bur} for $d = 1$, 

\smallskip
\item
in~\cite[Proposition 3.6]{serre}, \cite{duncan} for $d = 2$,
and 

\smallskip
\item
in~\cite{duncan-a7}, \cite{beauville}, \cite{prokhorov17} for $d = 3$. 
\end{itemize}

\smallskip
\noindent
For an overview, see~\cite[Section 6]{icm}.
This paper is in a similar spirit, with Theorem~\ref{thm.jordan} used in place of the above-mentioned classifications.

As a consequence of Theorem~\ref{thm1}(a), we will obtain the following.

\begin{theorem} \label{thm2} Let $\ZZ/ n\ZZ$ be a cyclic group of order $n$ and $H_n$ be a subgroup of $\Aut(\ZZ / n \ZZ) = (\ZZ/n \ZZ)^*$ 
for $n = 1, 2, 3, \ldots$. If $\lim_{n \to \infty} \, |H_n| = \infty$, then
$\lim_{n \to \infty} \, \ed_k((\ZZ/n \ZZ) \rtimes H_n) = \infty$ for any field $k$ of characteristic $0$.
\end{theorem}

In particular, $\lim_{n \to \infty} \, \ed_{\CC}((\ZZ/n \ZZ) \rtimes (\ZZ/ n \ZZ)^*) = \infty$. 
In the case where $n = p$ is a prime, all 
Sylow subgroups of $(\ZZ/p \ZZ) \rtimes (\ZZ/ p \ZZ)^*$ are cyclic, 
so~\eqref{e.sylow} reduces to the vacuous lower bound \[ \ed_{\CC}((\ZZ/p \ZZ) \rtimes (\ZZ/ p \ZZ)^*) \geqslant 1. \]
It was not previously known that $\ed((\ZZ/p \ZZ) \rtimes (\ZZ/ p \ZZ)^*) > 3$ for any prime $p$.

\section{Proof of Theorem~\ref{thm1}}
\label{sect2}

Let $G \to \GL(V)$ be a faithful linear representation of $G$. By the definition of essential dimension there 
exists a $G$-equivariant dominant rational map $V \dasharrow X$ such that $G$ acts faithfully on $X$ and $\dim(X) = \ed(G) \leqslant n$. 
By Theorem~\ref{thm.jordan}, there exists a normal abelian subgroup $A \triangleleft G$ such that $[G:A] \leqslant j(n)$. 

As we mentioned above, when $A$ is abelian, and $k$ has a primitive $e$th root of unity, we have
$\rdim(A) = \ed(A)$. Since $A$ is a subgroup of $G$, $\ed(A) \leqslant \ed(G) \leqslant n$. Thus there exists 
a faithful representation $W$ of $A$ of dimension $d \leqslant n$. The induced 
representation $V = \Ind_{A}^G(W)$ of $G$ is clearly faithful, and $\dim(V) = d [G: A] \leqslant n j(n) = r(n)$. 
Thus $\rdim(G) \leqslant \dim(V) \leqslant r(n)$. This proves (a).

To prove (b), note that in some basis $e_1, \dots, e_d$ of $W$, $A$ acts on $W$ by diagonal matrices. Choosing a set of representatives
$g_1, \dots, g_s$ for the cosets of $A$ in $G$, we see that the vectors $g_i e_j$ form a basis of $V$, as $i$ ranges from $1$ to $s$ and $j$ ranges from
$1$ to $d$. The group $G$ permutes the lines $\Span_k(g_i e_j)$ in $V$. The subgroup of $\GL(V)$ that fixed each of these lines individually is a maximal
torus $T = \GG_m^{ds}$. The subgroup of $\GL(V)$ that preserves this set of lines is the normalizer $N$ of $T$ in $\GL(V)$, where
$N \simeq T \rtimes \Sym_{ds}$. Thus our faithful representation $G \to \GL(V)$  embeds $G$ in $N$.
Since $s = [G: A] \leqslant j(n)$ and thus $ds \leqslant n j(n) = r(n)$, we can further 
embed $N$ into $(\GG_m)^{r(n)} \rtimes \Sym_{r(n)}$.
\qed

\begin{remark} Without the assumption that $k$ contains a primitive root $\zeta_e$ of unity of degree $e$,
our proof of Theorem~\ref{thm1}(a) only shows that if there exists a number $a_k(n)$ such that 
\[ \ed_k(A) \leqslant n \quad \Longrightarrow \quad \rdim_k(A) \leqslant a_k(n)  \]
for every finite abelian group $A$, then
\[ \ed_k(G) \leqslant n \quad \Longrightarrow \quad \rdim_k(G) \leqslant a_k(n) j(n).  \]
Note that $\Bir(X)(k) \subset \Bir(X)(\overline{k})$, where $\overline{k}$ is the algebraic closure of $k$,
we may assume that the Jordan constant $j(n)$ is the same for $k$ as for $\overline{k}$. On the other hand,
$a_k(n)$ may depend on $k$. Moreover, if $k$ is an arbitrary field of characteristic $0$, 
we do not know whether or not $a(n)$ exists. 
For example, if $p$ is a prime, then $\rdim_{\QQ}(\ZZ/p \ZZ) = p -1$ is not bounded from above, 
as $p$ increases, but it is not known whether or not $\ed_{\QQ}(\ZZ/p \ZZ)$ is bounded from above.
\end{remark}

\begin{remark} Finite subgroups of $\GG_m^2 \rtimes S$, for certain small finite groups $S$ play a prominent role in the classification of
finite groups of essential dimension $2$ (over $\CC$), due to A.~Duncan; see \cite[Theorem 1.1]{duncan}. Theorem~\ref{thm1}(b) suggests that this is not
an accident.
\end{remark}

\begin{remark} In the definition of the Jordan group we could have dropped the assumption that $A$ is normal: $\Gamma$ is Jordan if and only if there exists an integer $\tilde{j}$ such that every finite subgroup $G \subset \Gamma$ contains an abelian subgroup $A \subset G$ of index $[G: A] < \tilde{j}$. 
One usually refers to $j$ and $\tilde{j}$ as {\em the Jordan constant} and {\em the weak Jordan constant} for $\Gamma$, respectively. 
These constants are related by the inequalities $\tilde{j} \leqslant j \leqslant \tilde{j}^2$; see~\cite[Remark 1.2.2]{ps17}.
Indeed, if $G$ has an abelian subgroup of index $\leqslant i$, then it has a normal abelian subgroup of index $\leqslant i^2$; see~\cite[Theorem 1.41]{isaacs}. 

Now observe that our proof of Theorem~\ref{thm1}(a) does not use the fact that $A$ is normal. Thus we could have
defined $r(n)$ as $n \tilde{j}(n)$, rather than $n j(n)$, in the statement of Theorem~\ref{thm1}(a).
This will not make a difference in this paper, but may be helpful if one tries to find an explicit value for $r(n)$ for some
(or perhaps, even all) $n$. The constants $j(n)$ and $\tilde{j}(n)$ are largely mysterious, but some explicit values 
for $n = 2$ and $3$ can be found in~\cite{ps17}.
\end{remark}

\begin{remark} It is not true that $[G: Z(G)]$ is bounded from above, as $G$ ranges over the groups of essential dimension $\leqslant n$.
Here $Z(G)$ denotes the center of $G$. For example let $D_{2n}$ be the dihedral group of order $2n$. 
Then $\ed_{\CC}(D_{2n}) = 1$ for every odd integer $n$ (see~\cite[Theorem 6.2]{bur}), but $Z(D_{2n}) = 1$, and thus $[D_{2n}: Z(D_{2n})] = 2n$ 
is unbounded from above, as $n$ ranges over the odd integers.
\end{remark}

\section{Proof of Theorem~\ref{thm2}}
\label{sect3}

Let $l$ be the field obtained from $k$ by adjoining all roots of unity. Since \[ \ed_k(G) \geqslant \ed_{l}(G) \]
for every finite group $G$, we may replace $k$ by $l$ and thus assume that $k$ contains all roots of unity.
Under this assumption, we can restate Theorem~\ref{thm1}(a) as follows. Let $G_1, G_2, \dots$ be a sequence of finite groups. 
\begin{equation} \label{e.sequence}
\text{If $\lim_{n \to \infty} \, \rdim (G_n) = \infty$, then $\lim_{n \to \infty} \, \ed (G_n) = \infty$.} 
\end{equation}

The following lemma is elementary; we include a short proof for the sake of completeness.

\begin{lemma} \label{lem.rep} Let $q = p^a$ be a prime power,  
and $\phi \colon H \to \Aut(\ZZ/q \ZZ) = (\ZZ/ q \ZZ)^*$ be a group homomorphism.
Then $\rdim((\ZZ/ q \ZZ) \rtimes_{\phi} H) \geqslant |\phi(H)|$.
\end{lemma}

\begin{proof} Suppose $\rho \colon (\ZZ/ q \ZZ ) \rtimes_{\phi} H \to \GL(V)$ is a $d$-dimensional faithful representation. 
Our goal is to show that $d \geqslant |\phi(H)|$. By our assumption on $k$, $V$ as a direct sum of $1$-dimensional character spaces
$V = V_{\chi_1} \oplus \dots \oplus V_{\chi_d}$ for the cyclic group $\ZZ/ q \ZZ$, where the 
characters $\chi_1, \dots, \chi_d$ are permuted by $H$. Since $\rho$ is faithful,
the restriction of one of these characters, say of $\chi_i$, to the unique subgroup of order $p$ in $\ZZ/ q \ZZ$ is non-trivial.
Hence, $\chi_i \colon \ZZ/ q \ZZ \to k^*$ is faithful. This implies that the $H$-orbit of $\chi_i$ has exactly $|\phi(H)|$ elements. 
Thus $d \geqslant |\phi(H)|$, as desired.
\end{proof}

We are now ready to proceed with the proof of Theorem~\ref{thm2}. Set $G_n = (\ZZ/ n \ZZ) \rtimes H_n$.
By~\eqref{e.sequence},
it suffices to show that $\lim_{n \to \infty} \rdim(G_n) = \infty$.
In other words, for any positive real number $R$, we want to show that there are at most finitely many integers $n \geqslant 1$
such that 
\begin{equation} \label{e.rdim}
\rdim (G_n) \leqslant R. 
\end{equation}
Write 
\begin{equation} \label{e.H_n}
H_n = (\ZZ/ q_1 \ZZ)^{a_1} \times \dots \times (\ZZ/ q_r \ZZ)^{a_r}
\end{equation}
as a product of cyclic groups, where $q_1, \dots, q_r$ are distinct prime powers. 

\smallskip
Claim: If~\eqref{e.rdim} holds, then
(a) $a_i \leqslant R$, and (b) $q_i \leqslant R$,
for every $i = 1, \dots, r$. 

\smallskip
Assume for a moment that this claim is established.
For a fixed $R$, there are only finitely many groups of the form
$(\ZZ/ q_1 \ZZ)^{a_1} \times \dots \times (\ZZ/ q_r \ZZ)^{a_r}$
satisfying (a) and (b) (recall that $q_1, \dots, q_r$ are required to be distinct).
Since $\lim_{n \to \infty} \, |H_n| = \infty$, we conclude that the inequality $\rdim(G_n) \leqslant R$ holds 
for only finitely many integers $n \geqslant 1$, as desired.

It remains to prove the claim. For part (a), note that
\[ R \geqslant \rdim(G_n) \geqslant \rdim(H_n) \geqslant \rdim (\ZZ/q_i \ZZ)^{a_i}  = a_i \, . \]
To prove part (b), by symmetry it suffices to show that $q_1 \leqslant R$.
Let $n = p_1^{e_1} \dots p_s^{e_s}$ be the prime decomposition of $n$ and let $\phi_j \colon H_n \to \Aut(\ZZ/ p_j^{e_j} \ZZ)$ 
be the projection of 
\[ H_n \subset \Aut(\ZZ/ n \ZZ) = \Aut(\ZZ/ p_1^{e_1} \ZZ) \times \dots \times \Aut(\ZZ/ p_s^{e_s} \ZZ) \]
to the $j$th factor. Since $q_1$ is a prime power, at least one of
the projections $\phi_j$ maps the first factor $\ZZ/ q_1 \ZZ$ in~\eqref{e.H_n} 
isomorphically onto its image. Thus $G_n$ contains a subgroup isomorphic to
$(\ZZ / p_j^{e_j} \ZZ ) \rtimes_{\phi_j} (\ZZ/ q_1 \ZZ)$ and by Lemma~\ref{lem.rep},
\[ R \geqslant \rdim(G_n) \geqslant \rdim((\ZZ / p_j^{e_j} \ZZ) \rtimes_{\phi_j} (\ZZ/ q_1 \ZZ)) \geqslant |\phi_j(\ZZ / q_1 \ZZ)| = q_1 . \]
This completes the proof of the claim and thus of Theorem~\ref{thm2}.
\qed

\begin{example} \label{ex.brv} Fix a prime $p$. For each $n \geqslant 1$, choose a prime $q_n$ so that $q_n - 1$ is divisible by $p^n$.
There are infinitely many choices of such $q_n$ for each $n$ by Dirichlet's theorem of primes in arithmetic progressions. 
Embed $\ZZ/ p^n \ZZ$ into the cyclic group $(\ZZ/ q_n \ZZ)^*$ 
of order $q_n - 1$ and form the semidirect product $\Gamma_n = (\ZZ/q_n \ZZ_n) \rtimes (\ZZ/ p^n \ZZ)$.
% By Lemma~\ref{lem.rep}, $\rdim(\Gamma_n) = p^n - 1$. 

It is shown in~\cite[Example 3.5]{brv} that a conjecture of Ledet implies that $\ed_{\CC} (\Gamma_n) \geqslant n$.
Theorem~\ref{thm2} yields an unconditional proof of a weaker assertion: 
\[ \lim_{n \to \infty} \, \ed_{\CC} (\Gamma_n) =  \infty . \]
As is pointed out in~\cite{brv}, it was not previously known that $\ed_{\CC}(\Gamma_n) > 3$ for any $n$.
\end{example}

\section*{Acknowledgments} The author is grateful to Alexander Duncan and Fei Hu for helpful comments.

\end{document}